\documentclass[12pt]{article}
\usepackage{amsmath}
\usepackage{amsfonts}
\usepackage{amssymb}
\usepackage{amscd}
\usepackage{amsthm}
\usepackage{color}
\usepackage[matrix,arrow,curve]{xy}

\input xypic

\DeclareMathOperator{\IIm}{Im}
\DeclareMathOperator{\Ker}{Ker}

\newcommand{\ra}{{\rightarrow}}

\newcommand{\su}{\subset}
\newcommand{\mb}{\mathbb}
\newcommand{\nc}{\newcommand}
\nc{\al}{\alpha}

\newcommand{\fr}{\frac}

\theoremstyle{plain}
\newtheorem{theor}{Theorem}

\newtheorem{lemma}[theor]{Lemma}

\theoremstyle{remark}
\newtheorem{rmk}{Remark}
\newtheorem{examp}{Example}

\theoremstyle{definition}
\newtheorem{defin}{Definition}
\newtheorem{defin-prop}[theor]{Definition-Proposition}

\title{Conjugacy classes in discrete Heisenberg groups}
\author{Roman Budylin\\
\small{Steklov Mathematical Institute, Moscow, Russia}\\
\small{e-mail: {\tt budylin@mi.ras.ru}}}

\date{}

\begin{document}
\maketitle
\footnotetext[0]{
The author was supported by the grants MK 2856.2014.1, NSh-2998.2014.1, RFBR 13-01-12420, RFBR 14-01-00178, AG Laboratory NRU HSE, RF government grant, ag. 11.G34.31.0023.}
\begin{abstract}
We study some extension of a discrete Heisenberg group coming from the theory of loop-groups and find invariants of conjugacy classes in this group. In some cases including the case of the integer Heisenberg group we make these invariants more explicit.
\end{abstract}

Let us recall that a Heisenberg group $\rm{Heis} (3,R)$ over some ring $\rm R$ is the group of matrices
$$\begin{pmatrix}
1&n&c\\
0&1&p\\
0&0&1
\end{pmatrix},
$$
where $n,p,c\in \rm R$. When $\rm{R}=\mb Z$ we call such a group the integer Heisenberg group.
We shall use more general definition of the discrete Heisenberg group from \cite{PaAR}. The discrete Heisenberg group is the triple of abelian groups $(N,P,C)$ and a pairing $(\cdot,\cdot)$: $N\times P\ra C$.
The multiplication is the same as in the case of unipotent matrices:
$$\begin{pmatrix}
1&n&c\\
0&1&p\\
0&0&1
\end{pmatrix}
\begin{pmatrix}
1&n'&c'\\
0&1&p'\\
0&0&1
\end{pmatrix}
=\begin{pmatrix}
1&n+n'&c+c'+(n,p')\\
0&1&p+p'\\
0&0&1
\end{pmatrix}
$$
We would like to describe some algebraic geometry construction in which $\rm{Heis}(3,\mb Z)$ appears (cf.\cite{PaAR}).
Let $X$ be a surface over the field $k$. We fix a flag $P\in C$ on $X$ and assume that the point $P$ is a smooth point on $X$ and on the curve $C$. The 2-dimensional local field $K_{P,C}$ has the discrete valuation subring $\hat{O}_{P,C}$. It is mapped onto the local
field $k(C)_P$ on $C$. This local field contains his own discrete valuation subring $\hat{O}_P$ and
we denote its preimage in $\hat{O}_{P,C}$ by $\hat{O'}_{P,C}$. We set
$$\Gamma_{P,C} := K^*_{P,C}/\hat{O'}^*_{P,C}.$$
$\Gamma_{P,C}$ is (non-canonically) isomorphic to $\mb Z\oplus \mb Z$.
However, there is a canonical exact sequence of abelian groups
\begin{equation}
\begin{CD}
0 @>>> \mb Z @>>> \Gamma_{P,C} @>>>\mb Z @>>> 0.
\end{CD}
\label{locext}
\end{equation}
The map to $\mb Z$ in the sequence (\ref{locext}) corresponds to the discrete valuation $\nu_C$ with respect to $C$ and the subgroup $\mb Z$ corresponds to the discrete valuation $\nu_P$ on $C$ at $P$. A choice of
local coordinates $u$, $t$ in a neighborhood of $P$ such that locally $C = \{t = 0\}$ provides a
splitting of this exact sequence. The group $\Gamma_{P,C}$ will then be isomorphic to the subgroup $\{t^nu^m, n,m \in \mb Z\}$ in $K^*_{P,C}$. Change of a local parameter $t_C$ along the curve $C$ gives the automorphism of $\Gamma_{P,C}$ preserving extension (\ref{locext}).

There exists a well known central extension
\begin{equation}
\begin{CD}
1 @>>> k(C)^*_P @>i>> \tilde{K}^*_{P,C} @>\alpha>> K_{P,C}^* @>>> 1
\end{CD}
\label{ext2}
\end{equation}
with the map $\wedge^2 K_{P,C}^*\ra K(C)^*_P$
\footnote{this map is obtained by the commutator of liftings of two elements}
given by the tame symbol without sign $(f,g)\mapsto f^{\nu(g)}g^{-\nu(f)}\vline_C$. One can find an explicit description of this extension using torsors in \cite{Kac}.

Let us denote by $s$ a canonical section of extension (\ref{ext2}) over the subgroup $\hat{O}^*_{P,C}$. One may prove that the subgroup $s(\hat{O'}^*_{P,C}) i(\hat{O}^*_P)$ is normal in $\tilde{K}^*_{P,C}$. We put $\tilde{\Gamma}_{P,C}:=\tilde{K}_{P,C}/s(\hat{O'}^*_{P,C}) i(\hat{O}^*_P).$
Then we have the central extension
\begin{equation}
\begin{CD}
1 @>>> k(C)^*_P/\hat{O}^*_P @>>> \tilde{\Gamma}_{P,C} @>>> \Gamma_{P,C} @>>> 1.
\end{CD}
\label{ext3}
\end{equation}
The map obtained by the commutator of liftings for the extension (\ref{ext3}) is the standard symplectic form on $\mb Z^2$ so $\tilde{\Gamma}_{P,C}\simeq \rm{Heis}(3,\mb Z)$.

One can find in \cite{PaAR} the construction of some global Heisenberg-type groups corresponding to $X$. These groups encode a lot of  information about $X$, for example Chow groups can be recovered from them. Representation theory of discrete Heisenberg groups is studied in series of articles \cite{PaH},\cite{PaAR},\cite{PaA}.
It goes out more convenient to consider some extension of the Heisenberg group using a construction from the theory of loop groups. The point is that characters of irreducible representations in the extended Heisenberg group are not just distributions but functions. An important problem is to investigate central functions on the extended Heisenberg group and to obtain a Plancherel-type theorem. That is an explanation of our interest to conjugacy classes in the extended Heisenberg group. In this article we construct a full system of invariants, i.e. such set of invariants that two elements are conjugated if and only if values of invariants from the set coincide in these elements.  

This investigation originated from a question of A.N.Parshin and the author is very grateful to him for an encouragement. Âåsides I am thankfull to D.V.Osipov who thoroughly read the text and make a lot of essential remarks.
\section{The construction of extended Heisenberg group}
For abelian groups $P$ and $C$ we can consider the group $\rm{Aut_{gr}}(P\oplus C)$ of automorphisms which preserve the filtration $0\subset C\subset P\oplus C$ and acts trivially on its adjoint quotients (we call such automorphisms graded). The Heisenberg group can be represented as a semidirect product $(P\oplus C)\rtimes N$ where $N$ is a subgroup in $\rm{Aut_{gr}}(P\oplus C)$. Graded automorphism can be given by the homomorphism of $P$ to $P\oplus C$ such that the following diagram is commutative:
$$\xymatrix{
P\ar[r]\ar[dr]_{\rm{id}}&P\oplus C\ar[d]\\
&P
}
$$
So this homomorphism is the pair $(\rm{id},n:P\ra C)$. It implies that $\rm{Aut_{gr}}(P\oplus C)$ is abelian and that there is a pairing $$\rm{Aut_{gr}}(P\oplus C)\times P\ra C,$$ which gives us the pairing $N\times P\ra C$ from the definition of the Heisenberg group. $N$ acts on $P\oplus C$ in the following way:
$$n(p,c)=(p,c+n(p))$$
Now we examine a more general situation. Let $G=A\rtimes B$. We are interested in graded automorphisms of $G$, i.e. preserving the filtration $(1)\su A\su A\rtimes B$ and acting trivially on its adjoint quotients. Any such automorphism is defined by a homomorphism $\phi:B\ra G$ with some conditions. Let $\phi(b)=\alpha(b)b$, where $b\in B, \alpha(b)\in A$. What conditions on $\alpha$ should we apply to have graded automorphism?
\begin{lemma}
The map $\phi$ defines a graded automorphism $f$ of $G$ if and only if following two conditions hold:

\begin{enumerate}
\item $\alpha(b)$ belongs to $Z(A)$ -- the center of $A$.

\item $\alpha$ is a crossed homomorphism $B\ra Z(A)$, i.e. we have $\alpha(b_1)b_1\alpha(b_2)b_1^{-1}=\alpha(b_1b_2)$.
\end{enumerate}

\label{alpha}
\end{lemma}
\begin{proof}
One can prove that $f$ is an automorphism if and only if $\phi$ is a homomorphism preserving the conjugacy action of $B$ on $A$.
The second condition is equivalent to $\phi$ being a homomorphism:
$$\alpha(b_1)b_1\alpha(b_2)b_2=\alpha(b_1b_2)b_1b_2,$$
$$\alpha(b_1)b_1\alpha(b_2)b_1^{-1}=\alpha(b_1b_2).$$
The first condition is equivalent to $\phi$ preserving the conjugacy action of $B$ on $A$:
$$\phi(b)(a)=f(b(a))=b(a).$$
So $\alpha(b)$ commutes with every $a\in A$.
\end{proof}
Now we consider a group $(A\rtimes B)\rtimes \rm{Aut_{gr}}G$. The following definition is slightly more general than analogous from \cite{PaAR}.
\begin{defin}
\label{Heis}
An extended Heisenberg group is the group $((P\oplus C)\rtimes N)\rtimes K)$ for some subgroup $K$ in $\rm{Aut_{gr}}((P\oplus C)\rtimes N)$.
\end{defin}
Let us find crossed homomorphisms $k:N\ra P\oplus C$ satisfying conditions of lemma \ref{alpha}. Let $k(n)=(k_p(n),k_c(n))$. As $P\oplus C$ is abelian the first conditon is satisfied. Let us check out the second condition, let $k(n)=(k_p(n),k_c(n))$:
$$(k_p(n_1),k_c(n_1))(k_p(n_2),k_c(n_2)+n_1(k_p(n_2))=(k_p(n_1+n_2),k_c(n_1+n_2)),$$
$$(k_p(n_1)+k_p(n_2),k_c(n_1)+k_c(n_2)+n_1(k_p(n_2)))=(k_p(n_1+n_2),k_c(n_1+n_2)).$$
So $k_p:N\ra P$ is a homomorphism and for $k_c$ we have the condition:
\begin{equation}
k_c(n_1+n_2)=k_c(n_1)+k_c(n_2)+n_1(k_p(n_2))
\label{add}
\end{equation}
\begin{rmk}
\label{comm}
From (\ref{add}) it follows that $n_1(k_p(n_2))=n_2(k_p(n_1))$ because $N$ is abelian.
\end{rmk}
\begin{rmk}
Condition (\ref{add}) holds when we add some homomorphism $k_c^{\phi}:N\ra C$ to $k_c$.
\end{rmk}
\begin{examp}
\label{ex}
The map $k_c(n)=k\fr{n(n-1)}{2}:\mb Z\ra \mb Z$ satisfies condition (\ref{add}) for $k_p(n)=kn:\mb Z\ra \mb Z$. So the following action of $K=\mb Z$ on the integer Heisenberg group is graded:
$$k(p,c,n)=(p+kn,c+k\fr{n(n-1)}{2},n).$$
And definition \ref{Heis} gives the extended integer Heisenberg group.
\end{examp}
So an element $k$ of $\rm{Aut_{gr}}((P\oplus C)\rtimes N)$ is given by a pair $(k_p,k_c)$ where $k_p:N\ra P$ is a homomorphism and $k_c:N\ra C$ satisfies (\ref{add}). The action of $k$ on $(P\oplus C)\rtimes N$ is following:
$$k(p,c,n)=(p+k_p(n),c+k_c(n),n).$$
A composition of such actions has $k_p=k_p^1+k_p^2$ and $k_c=k_c^1+k_c^2$, so $\rm{Aut_{gr}}$ is abelian.
Let us study functions $k_c$ more closely.
\begin{lemma}
Let $N$ be a group $\mb Z/l_1\mb Ze_1\oplus \mb Z/l_2\mb Ze_2\oplus \dots \mb Z/l_s\mb Ze_s$ where $l_i\in \mb Z$ and let we have a homomorphism $k_p:N\ra P$. Then $k_c$ with condition (\ref{add}) exists if and only if:
\begin{enumerate}
\item
\label{c1} $n_1(k_p(n_2))=n_2(k_p(n_1))$ for any $n_1,n_2$.

\item
\label{c2}$\fr{l_i(l_i-1)}{2}e_i(k_p(e_i))\in l_iC.$
\end{enumerate}
If $k_c$ exists then it can be represented as:
$$k_c(m_1e_1+\dots+m_se_s)=\sum_{i<j}m_im_je_i(k_p(e_j))+\sum_i \fr{1}{2}m_i(m_i-1)e_i(k_p(e_i))+m_ix_i+k_c^{\phi}(m_1e_1+\dots+m_se_s),$$
where $m_i\in \mb Z$ and $k_c^{\phi}$ is a homomorphism from $N$ to $C$ and $x_i$ is any solution of the equation $l_ix_i+\fr{l_i(l_i-1)}{2}e_i(k_p(e_i))=0$.\end{lemma}
\begin{proof}
It's not difficult to prove that $k_c(0)$ equals zero.
Also one can proof using condition (\ref{add}) that
$$k_c(ne_i)=nk_c(e_i)+\fr{n(n-1)}{2}e_ik_p(e_i).$$
This formula gives us a map $k'_c(n)$ from $\mb Z$ to $N$ satisfying condition (\ref{add}) for $k_p'(n)=k_p(ne_i)$. The map $k_c'$ comes from some $k_c$ with conditon (\ref{add}) from $\mb Z/ l_i\mb Z e_i$ to $N$ if and only if $k_c'(l_i)=0$. So we need condition \ref{c2} to define $k_c$ on $\mb Z /l_i \mb Ze_i$ . Condition \ref{c1} is necessary because $k_c(n_1+n_2)= k_c(n_2+n_1)$.

Now suppose that conditions \ref{c1} and \ref{c2} hold and let us prove that a function $k_c$ with required properties exists. Define $x_i\in C$ as any solution of the equation $l_ix_i+\fr{l_i(l_i-1)}{2}e_i(k_p(e_i))=0$. Consider a function $k''_c:N\ra C$ defined by the formula
$$k''_c(m_1e_1+\dots+m_se_s)=\sum_{i<j}m_im_je_i(k_p(e_j))+\sum_i \fr{m_i(m_i-1)}{2}e_i(k_p(e_i))+m_ix_i.$$
Condition \ref{c1} implies that this formula doesn't depend on the order of summands. Also we have that $\fr{m_i(m_i-1)}{2}e_i(k_p(e_i))+m_ix_i=\fr{(m_i+l_i)(m_i+l_i-1)}{2}e_i(k_p(e_i))+(m_i+l_i)x_i$ because of condition \ref{c2} and $l_ik_p(e_i)=0$. So the function $k''_c$ is correctly defined.
One can prove that it satisfies condition (\ref{add}). At last if there exists some $k_c$, then the function $k_c-k''_c$ is just a homomorphism. This proves the last statement.
\end{proof}
\begin{rmk}
\label{2kc}
For any function $\gamma \colon N\ra C$ we can consider its polarization $\hat{\gamma}(n_1,n_2)=\gamma(n_1+n_2)-\gamma(n_1)-\gamma(n_2)$. It is clear that if polarizations of two functions are equal then difference of this functions is a homomorphism $N\ra C$. Polarizations of $2k_c(n)$ and $nk_p(n)$ are equal so $2k_c(n)=\phi_k(n)+n(k_p(n)),$ where $\phi_k(n)$ is a homomorphism $N\ra C$ depending from $k$.
\end{rmk}

\section{Conjugacy classes in a Heisenberg group}
Consider an extended Heisenberg group $G=((P\oplus C)\rtimes N)\rtimes K)$. The formula of multiplication in this group is the following:
\begin{multline*}
(p',c',n',k')*(p_1,c_1,n_1,k_1)=\\
(p'+p_1+k'_p(n_1),c_1+c'+k'_c(n_1)+n'(p_1+k'_p(n_1)),n_1+n',k_1+k')
\end{multline*}
Let $x_2=(p_2,c_2,n_2,k_2)$ be conjugated to $x_1=(p_1,c_1,n_1,k_1)$ by means of $(p',c',n',k')$. It is equivalent to
\begin{multline*}
(p'+p_1+k'_p(n_1),c_1+c'+k'_c(n_1)+n'(p_1+k'_p(n_1)),n_1+n',k_1+k')=\\
=(p_2+p'+k_{2,p}(n'),c_2+c'+k_{2,c}(n')+n_2(p'+k_{2,p}(n')),n_2+n',k_2+k')
\end{multline*}
So we get that $n_1=n_2=n$, $k_1=k_2=k$ and
\begin{equation}
\label{sis1}
\left\{
\begin{aligned}
p_1-p_2&=k_p(n')-k'_p(n)\\
c_1+k'_c(n)+n'(p_1+k'_p(n)) & =c_2+k_c(n')+n(p'+k_p(n')).\\
\end{aligned}
\right.
 \end{equation}
Using the summand $n(p')$ in the second equation of (\ref{sis1}) and Remark \ref{comm} we can reduce the second equation $\rm{mod} \IIm n$, where $n:P\ra C$ is a homomorphism given by fixed $n\in N$. So we get
\begin{equation}
\label{sis2}
\left\{
\begin{aligned}
p_1-p_2&=k_p(n')-k'_p(n)\\
c_1-c_2&=k_c(n')-k'_c(n)-n'(p_1) \pmod{\IIm n}.\\
\end{aligned}
\right.
 \end{equation}
Now we fix $n$ and $k$ untill the end of the next section and obtain the system of equations on $n',k'$ with coefficients depending on $p_i,c_i$.

Let us denote by $\Lambda$ the homomorphism $(n',k')\mapsto k_p(n')-k'_p(n)$ from $N\oplus K$ to $P$. It is obvious that $R((p,c,n,k))=p\in P/\IIm \Lambda$ is an invariant of a conjugacy class of $(p,c,n,k)$.
Suppose that
\begin{equation}
\text{ all considered elements have the same value $r$ of the invariant $R$.}
    \label{cond}
\end{equation}
 Let us denote by $B_{12}$ the map $(n',k')\mapsto k_c(n')-k'_c(n)-n'(p_1)$ from $N\oplus K$ to $C /\IIm n$. Let us notice that $B_{12}$ depends only on the first index. Now we consider sets $V_{12}:=c_1-c_2-B_{12}(\Lambda^{-1}(p_1-p_2))\subset C /\IIm n$. Indices $12$ mean that $B_{12}$ and $V_{12}$ correspond to the pair $(x_1,x_2)$. In the same way we can define $V_{ij}$ when we have a collection $\{x_i\in G,i\in I\}$.
\begin{lemma}
\label{sol}
Under condition (\ref{cond}) we have that

1. $V_{12}=-V_{21}$

2. $V_{12}+V_{23}=V_{13}$
\end{lemma}
\begin{proof}
\begin{enumerate}
\item
 Let $n',k'$ be elements from $\Lambda^{-1}(p_1-p_2)$. Then $-n',-k'$ are elements from $\Lambda^{-1}(p_2-p_1)$ and it is sufficient to prove that
$B_{12}(n',k')=-B_{21}(-n',-k')$. We have the following chain of equalities:
\begin{multline*}
B_{12}(n',k')+B_{21}(-n',-k')=(k_c(n')-k'_c(n)-n'(p_1))+k_c(-n')+k'_c(n)+n'(p_2)=\\
=k_c(n')+k_c(-n')+n'(p_2-p_1)=k_c(0)-n'(k_p(-n'))+n'(k'_p(n)-k_p(n'))=0
\end{multline*}
Here we use definitions of $B_{ij}$ and $\Lambda$, the equality $n'(k'_p(n))=n(k'_p(n'))=0$ in $C/\IIm n$, the equality $k_c(0)=0$ and the equation (\ref{add}).

 \item
 Let $n'_1,k'_1$ and $n'_2,k'_2$ be elements of $\Lambda^{-1}(p_1-p_2)$ and $\Lambda^{-1}(p_2-p_3)$ correspondingly. Then we have
\begin{multline*}B_{12}(n'_1,k'_1)+B_{23}(n'_2,k'_2)=(k_c(n'_1)-k'_{1,c}(n)-n'_1(p_1))+(k_c(n'_2)-k'_{2,c}(n)-n'_2(p_2))=\\
=(k_c(n'_1)+k_c(n'_2))+n'_2(p_1-p_2)-(k'_{1,c}+k'_{2,c})(n)-(n'_1+n'_2)(p_1)=\\
=(k_c(n'_1+n'_2)-n'_1(k_p(n'_2)))+n'_2(k_p(n'_1)-k'_{1,p}(n))-(k'_{1,c}+k'_{2,c})(n)-(n'_1+n'_2)(p_1)=\\
=-n'_1(k_p(n'_2))+n'_2(k_p(n'_1)-k'_{1,p}(n))+B_{13}(n'_1+n'_2,k'_1+k'_2)=B_{13}(n'_1+n'_2,k'_1+k'_2)
\end{multline*}
In the fourth equality we use the first equation from system (\ref{sis2}) which holds because of condition (\ref{cond}). Also we use that $n'(k'_p(n))=n(k'_p(n'))=0$ in the fifth equality.
So from this chain of equalities we get that $V_{12}+V_{23}\su V_{13}$. Using that $V_{23}=-V_{32}$ we get the inverse inclusion.

\end{enumerate}
\end{proof}
\begin{rmk}
\label{Vrem}
Let us notice that $V=B_{12}(\Ker \Lambda)$ doesn't depend on a pair of elements. It depends only on the value of $R$. Indeed, we have $B_{12}(\Ker\Lambda)=V_{11}$ and by lemma \ref{sol} $V_{11}=V_{12}+V_{21}=V_{21}+V_{12}=V_{22}=B_{23}(\Ker \Lambda)$. So $V$ is also an invariant of a conjugacy class. We have that $V_{11}+V_{11}=V_{11}$ and $0\in V_{11}$ because $V_{11}=-V_{11}$. So $V$ is a group.
\end{rmk}

\begin{lemma}
$V_{12}$ is a coset in the group $C /\IIm n$ by the subgroup $B_{12}(\Ker \Lambda).$
\end{lemma}
\begin{proof}
By remark \ref{Vrem} we have that $V_{11}$ is the group. If we fix some element $x$ in $V_{12}$ then by lemma \ref{sol} we have that $x-V_{12}\su V_{11}$ and from the other hand $x+V_{11}\su V_{12}$ so $x-V_{12}=V_{11}=B_{12}(\Ker \Lambda)$.
\end{proof}
$V_{12}$ gives us an element $v_{12}$ in $C/V$. And the second equation of the system (\ref{sis2}) is equivalent to $v_{12}$ being $0$, because in this case $0\in V_{12}$. Let us fix some element $(p_0,c_0,n,k)$ with $R((p_0,c_0,n,k))=r$ and set its index equal to $0$, this is a base for construction of our last invariant. Denote by $S(x_1)$ the element $v_{01}$. We have the equality $v_{01}+v_{12}=v_{02}$. Then $v_{12}=v_{10}-v_{20}$. So the second equation of the system (\ref{sis2}) is equivalent to $S(x_1)=S(x_2)$ and $S(x)$ is an invariant of a conjugacy class. So we have the set of invariants $\{n,k,R,S\}$, where $n\in N,k\in K, R\in P/\IIm \Lambda$ and $S\in C/V$. It is a full system of invariants because if invariants of two elements coincide then system (\ref{sis2}) is solvable.
\section{The case of odd $C/\IIm n$ and $\phi_k=0$.}
If $C/\IIm n$ is finite of an odd order then we can simplify the system (\ref{sis2}). By remark \ref{2kc} in this case we have following equalities
\begin{multline*}
k_c(n')=\fr{1}{2}(2k_c(n'))=\fr{1}{2}(\phi_k(n')+n'(k_p(n')))=\fr{1}{2}(\phi_k(n')+n'(p_1-p_2+k'_p(n)))=\\
=\fr{1}{2}(\phi_k(n')+n'(p_1-p_2))\mod \IIm n,
\end{multline*}
where $\phi_k$ is the homomorphism from $N$ to $C$. We use the first equation from system (\ref{sis2}) in the third equality. Besides, $k'_c(n)=\fr{1}{2}\phi_{k'}(n)\mod \IIm n$. Now we suppose that $\phi_s=0$ for all $s\in K$. Let us denote by $K(n)$ the group $\{k_p(n),k\in K\}$ for fixed $n$. Then we get the new system
\begin{equation}
\label{sis3}
\left\{
\begin{aligned}
p_1-p_2&=k_p(n')\pmod{K(n)}\\
2c_1-2c_2&=-n'(p_1+p_2)\pmod{\IIm n}.\\
\end{aligned}
\right.
\end{equation}
First invariant here is $R=p_1\bmod{(K(n)+k_p(N))}$. Now we fix a value $r$ of $R$. Let us consider for any $k\in K$ the homomorphism $\tilde{k_p}$ which is the composition of $k_p$ and the natural homomorphism from $P$ to $P/K(n)$. In notations from the previous section using divisibility by $2$ of elements from $C/\IIm n$ we get $V'_{12}=2V_{12}=2c_1-2c_2+(\tilde{k_p}^{-1}(p_1-p_2))(p_1+p_2)\subset C/\IIm n$. The second equation of (\ref{sis3}) is equivalent to $0\in V'_{12}$. As before we have that $V'_{12}+V'_{23}=V'_{13}$ and that $V'_{12}$ is the coset of $B_{12}(Ker \Lambda)=\Ker( \tilde{k_p}(r))$. And as the last invariant we can take $V'_{10}\in C/\Ker(\tilde{k_p}(r))$ where element with index $0$ is some fixed element whose value of $R$ equals $r$.

There is an interesting case when $\tilde{k_p}$ is surjective in $P/K(n)$. In this case the set $\tilde{k_p}^{-1}(p_1)(p_2)$ is a singleton in $C/\IIm n$ because $\tilde{k_p}^{-1}(0)(k_p(n''))=0 \pmod{\IIm n}$ for any $n''\in N$ and we have that $\tilde{k_p}^{-1}(p_1)(p_2)=\tilde{k_p}^{-1}(p_2)(p_1)$. Therefore, the second equation is equivalent to $2c_1+\tilde{k_p}^{-1}(p_1)(p_1)=2c_2+\tilde{k_p}^{-1}(p_2)(p_2) \pmod{\IIm n}$. So the full system of invariants in this case is $\{n \in N,k \in K,p \bmod{(K(n)+k_p(N))}\in P/(K(n)+k_p(N)),2c+\tilde{k_p}^{-1}(p)(p)\in C/\IIm n\}$.

\section{Extended integer Heisenberg group}
Now we consider the case of the Heisenberg group from example \ref{ex}. So we assume that $N,P,C,K$ are all isomorphic to $\mb Z$, pairings $N\times P\ra C$ and $N\times K\ra P$ are just multiplication and $k_c(n)=k\fr{n(n-1)}{2}$.
We need the following lemma.
\begin{lemma}
Suppose that we have a system of integer equations
\begin{equation}
\left\{
\begin{aligned}
ax&=b \pmod{n}\\
cx&=d \pmod{n}.\\
\end{aligned}
\right.
\end{equation}
Then it has a solution $x\in \mb Z$ if and only if $(a,n)|b$ and $d(a,n)-bcw=0\pmod {n(a,c,n)}$, where $(a,n):=gcd(a,n),(a,c,n):=gcd(a,c,n)$ and $w$ is any solution of the equation $\fr{a}{(a,n)}w=1\pmod{\fr{n}{(a,n)}}$.
\label{solv}
\end{lemma}
\begin{proof}
It is clear that the condition $(a,n)|b$ is equivalent to the first equation being solvable. If $(a,n)|b$ then the first equation can be rewritten as $a'x=b'\pmod {\fr{n}{(a,n)}}$, where $a',b'$ are $a,b$ divided by $(a,n)$. So we have that $x=wb'+i\fr{n}{(a,n)}$, where $i\in \mb Z$ and $w$ is any solution of the equation $\fr{a}{(a,n)}w=1\pmod{\fr{n}{(a,n)}}$. Placing this expression in the second equation we get the equation
$$cb'w-d=-c\fr{n}{(a,n)}i\pmod{n}.$$
It is solvable if and only if $cb'w-d$ is divisible by $(c\fr{n}{(a,n)},n)$. Multiplying by $(a,n)$ we get the condition $bcw-d(a,n)=0\pmod{n(a,c,n)}$ because $(cn,n(a,n))=n(c,(a,n))=n(a,c,n)$.
\end{proof}
Let $(p_1,c_1,n_1,k_1)$ be conjugated with $(p_2,c_2,n_2,k_2)$. Then
we have that $n_1=n_2=n$, $k_1=k_2=k$ and there is the system similar to system (\ref{sis2})
\begin{equation}
\label{sis5}
\left\{
\begin{aligned}
p_1-p_2&=kn'-k'n\\
c_1-c_2&=k\fr{n'(n'-1)}{2}-k'\fr{n(n-1)}{2}-n'p_1 \pmod{n}.\\
\end{aligned}
\right.
 \end{equation}

Now we consider two cases: $n$ is odd and $n$ is even.

At first, suppose that $n$ is odd.
Then as in the previous section we get the system:
\begin{equation}
\label{sis4}
\left\{
\begin{aligned}
p_1-p_2&=kn'\pmod{n}\\
(2c_1+p_1)-(2c_2+p_2)&=-(p_1+p_2)n'\pmod{n}.\\
\end{aligned}
\right.
\end{equation}
Using lemma \ref{solv} we get that it is solvable if and only if $p_1=p_2\pmod{(k,n)}$ and $-p_1^2w-(n,k)(2c_1+p_1)=-p_2^2w-(n,k)(2c_2+p_2)\pmod{n(n,k,p)}$, where $w$ is any fixed solution of the equation $\fr{k}{(k,n)}w=1\pmod{\fr{n}{(n,k)}}$. So in this case  $\{n,k,p\bmod{(n,k)},-p^2w-(n,k)(2c+p)\bmod{n(n,k,p)}\}$ is a full system of invariants.

Now suppose that $n$ is even.
Multiplying by $2$ the second equation of (\ref{sis5}) and expressing $kn'$ from the first equation we get the system:
\begin{equation}
\label{sis6}
\left\{
\begin{aligned}
p_1-p_2+nk'&=kn'\pmod{2n}\\
(2c_1+p_1)-(2c_2+p_2)&=-(p_1+p_2+nk')n' \pmod{2n}.\\
\end{aligned}
\right.
 \end{equation}
We can consider the first equation above $\pmod{2n}$ because if $(k',n')$ satisfy the second equation then $(k'+2l,n')$ for any $l\in \mb Z$ also satisfy it. So if you have solution $(k',n')$ of the system (\ref{sis6}) you can find such $l\in \mb Z$ that $(k'+2l,n')$ satisfy the first equation of system (\ref{sis5}).   
The equality $n^2=0\pmod{2n}$ is used above (it holds because $n$ is even).

We consider two cases $k'\equiv 0\pmod{2}$ and $k'\equiv 1\pmod{2}$. Let $w$ be any fixed solution of the equation $\fr{k}{(k,2n)}w=1\pmod{\fr{2n}{(2n,k)}}$. In the first case the solvability by lemma \ref{solv} is equivalent to
\begin{equation}
\label{0eqv}
\left\{
\begin{aligned}
p_1&=p_2\pmod{(2n,k)}\\
-p_1^2w-(2n,k)(2c_1+p_1)&=-p_2^2w-(2n,k)(2c_2+p_2)\pmod{2n(2n,k,2p_1)}.\\ 
\end{aligned}
\right.
 \end{equation}
In the second case the solvability is equivalent to
\begin{equation}
\label{1eqv}
\left\{
\begin{aligned}
p_1&=p_2+n\pmod{(2n,k)}\\
(-p_1^2w-(2n,k)(2c_1+p_1))-wn(2p_1+n)&=\\
=(-p_2^2w-(2n,k)(2c_2+p_2))& \pmod{2n(2n,k,2p_1)},\\
\end{aligned}
\right.
 \end{equation}
We call the first case 0-equivalence and the second case 1-equivalence. Two elements are conjugated if and only if they are 0- or 1- equivalent. Composition of $\delta_1$-equivalence and $\delta_2$-equivalence is $(\delta_1+\delta_2)$-equivalence, where $\delta_1,\delta_2 \in \mb Z/2\mb Z$.
Let us introduce following notations $I_1((p,c,n,k)):= p\bmod{(2n,k)},I_2((p,c,n,k)):=p+n\bmod{(2n,k)},J_1((p,c,n,k)):=-p_1^2w-(2n,k)(2c_1+p_1),
J_2((p,c,n,k)):=-p_1^2w-(2n,k)(2c_1+p_1)-wn(2p_1+n)$. Above we prove that if $x_1$ and $x_2$ are 0-equivalent then $I_1(x_1)=I_1(x_2)$ and $J_1(x_1)=J_1(x_2)$, when $x_1$ and $x_2$ are 1-equivalent then $I_2(x_1)=I_1(x_2)$ and $J_2(x_1)=J_1(x_2)$. And it is easy to prove that in the first case we have also that $I_2(x_1)=I_2(x_2)$ and $J_2(x_1)=J_2(x_2)$, in the second case $I_1(x_1)=I_2(x_2)$ and $J_1(x_1)=J_2(x_2)$. So we get that the set $\{J_1(x),J_2(x)\}$ and the set $\{I_1(x),I_2(x)\}$ are invariants of a conjugacy class. So $\{n,k,\{I_1(x),I_2(x)\},\{J_1(x),J_2(x)\}\}$ is the full system of conjugacy class invariants.

\end{document}